\documentclass{amsart}

\usepackage{amsthm,amsmath,amssymb,amsfonts}
\usepackage[all]{xy}
\usepackage{hyperref}
\usepackage{graphicx}
\usepackage{mathrsfs}
\usepackage{amsaddr}

\usepackage{biblatex}
\newtheorem{thm}{Theorem}[section]
\newtheorem{prop}[thm]{Proposition}
\newtheorem{cor}[thm]{Corollary}
\newtheorem{lem}[thm]{Lemma}

\theoremstyle{definition}
\newtheorem{defn}[thm]{Definition}

\theoremstyle{remark}
\newtheorem{ex}[thm]{Example}

\newcommand{\cat}[1]{\mathfrak{#1}}

\newcommand{\set}{\mathbf{Set}}
\newcommand{\qord}{\mathbf{QOrd}}

\newcommand{\wsetc}{\mathbf{WSet_{1}}}
\newcommand{\schc}{\mathbf{Sch_{1}}}
\newcommand{\metc}{\mathbf{Met_{1}}}
\newcommand{\wsetg}{\mathrm{w}\mathbf{Set}}
\newcommand{\metg}{\mathrm{d}\mathbf{Mtr}}

\newcommand{\digra}{\mathbf{Digra}}
\newcommand{\ssys}{\mathbf{SSys}}
\newcommand{\gra}{\mathbf{Gra}}
\newcommand{\sdigra}{\mathbf{SDigra}}
\newcommand{\istr}{\mathbf{IStr}}

\newcommand{\spa}{\mathbf{Spa}}

\DeclareMathOperator{\card}{card}
\DeclareMathOperator{\ob}{Ob}
\DeclareMathOperator{\del}{Del}

\title[Simplification \& Incidence]{
Simplification \& Incidence:
How an Incidence-focused Perspective Patches Category-theoretic Problems in Graph Theory
}
\author{Will Grilliette}
\address{National Security Agency, Fort George G Meade MD, MD 20755-6844, USA}

\begin{document}

\begin{abstract}
By applying simplification operations to categories of multigraphs,
several natural graph operations are shown to demonstrate categorical issues.
The replacement of an undirected edge with a directed cycle for digraphs admits both a left and a right adjoint,
while the analogous operation for quivers only admits a left adjoint.
The clique-replacement graph,
intersection graph,
and dual hypergraph fail to be functorial with traditional graph homomorphisms.
The three failures are remedied by considering weak set-system homomorphisms,
which form a category isomorphic to both the category of incidence structures and a lax comma category.
\end{abstract}

\maketitle

\section{Introduction}

In \cite{ih3,grilliette2023},
several issues were found in the classical categories of graphs and hypergraphs.
These issues were addressed by creating the category of incidence hypergraphs,
a presheaf topos,
and utilizing its structure.
This paper continues by applying the simplification operators from \cite{simplification1} to transfer the categorical properties from a category of graphs to its simplified counterpart.
In so doing,
many parallels arise naturally,
such as deletion of nontraditional edges for hypergraphs and set systems
and replacing undirected edges with directed cycles for quivers and digraphs.

However,
while these operations are quite similar and greatly parallel each other,
they behave quite differently.
For instance,
the operation of replacing undirected edges with directed cycles for digraphs admits both a right and a left adjoint adjoint,
acting as the embedding of a reflective and coreflective subcategory,
but the analogous operation for quivers only admits a left adjoint.
Thus,
one cannot merely regard similar graph-like structures as the same.
Such intuitive reasoning can lead one to fallacy,
which becomes more evident when considering other graph operations,
such as the clique replacement graph,
intersection graph,
and the dual hypergraph.
Each of these three operations is proven not to be a functor when using traditional graph homomorphisms,
contrary to the claims in \cite[p.\ 187-188]{dorfler1980}.

Fortunately,
each of these issues can be remedied.
First,
each type of graph-like structure is compartmented into its own category to prevent any potential conflation between them,
and how these structures interact is studied.
Many of them are similar,
but the finer details highlight important distinctions,
such as presence of adjoints.
Second,
a weakened version of graph homomorphism is introduced and proven to be intimately connected with the notions related to incidence hypergraphs.
Surprisingly,
the category of incidence structures is isomorphic to the category of set-system hypergraphs with these weak homomorphisms,
and the isomorphism is precisely the incidence-forming and incidence-forgetting operations from \cite[p.\ 22-23]{grilliette2023},
but viewed through the lens of weak homomorphisms.
Finally,
the clique-replacement graph,
intersection graph,
and dual hypergraph are revisited and shown to operate far more pleasantly under the incidence-theoretic paradigm than under the adjacency perspective.

The above results do beg an important question.
Several natural graph operations do not behave well with traditional graph homomorphisms,
but do when applied to incidence-preserving maps.
Moreover,
these graph operations appear to be natural incidence-theoretic operations obfuscated under layers of conversions between different types of graphs.
Indeed,
the incidence-forming and incidence-forgetting operators from \cite[p.\ 22-23]{grilliette2023} are an isomorphism of categories,
but blurred by simplification and use of traditional graph homomorphisms.
The clique-replacement graph manifests naturally from a kernel pair in the category of sets,
but is buried under traditional graph homomorphisms and simplification operations.
The intersection graph is just the incidence dual from \cite[p.\ 13]{ih3},
but masked by simplicial replacement,
simplification,
and deletion.
Are the extra conversions necessary?
Perhaps it would be simpler and more elegant to work directly with the core operations,
rather than digging through layers of obfuscation to perform a task.

Section \ref{sec:fallacy} addresses the concerns of the polysemous usage of ``graph'' and,
more worrisome,
the fallacy of following intuition without rigorous proof.
Section \ref{sec:incidenceview} considers a view of graph homomorphisms through incidence,
rather than adjacency,
introducing weak homomorphisms of hypergraphs.
Section \ref{sec:cliquereplacement} revisits the clique-replacement construction,
now viewed through the weak homomorphisms of Section \ref{subsec:weakhomomorphisms}.
Section \ref{sec:duality} returns to the dual hypergraph,
showing that it is functorial when considering weak homomorphisms.
Using the dual hypergraph,
the intersection graph can be reformed by inserting the duality inside the composition forming the modified clique replacement.

Before progressing into the main points,
some preliminaries are recalled to reference concepts and set notation,
which will be used throughout the paper.
These topics include the category of sets,
the category of quasi-ordered sets,
and the lax comma category.

\subsection{Preliminaries}

This section covers some preliminaries,
which will be used throughout this paper.
Section \ref{subsubsec:sets} discusses three key functors on the category of sets:
the diagonal functor $\Delta$,
its right adjoint $\Delta^{\star}$,
and the covariant power-set functor $\mathcal{P}$.
Section \ref{subsubsec:quasiordered} likewise describes four important functors on the category of quasi-ordered sets:
the underlying-set functor $F$,
its left adjoint $F^{\diamond}$,
the covariant power-set functor $\mathcal{P}^{+}$,
and the dual quasi-order autofunctor $\Box^{\mathrm{T}}$.
Section \ref{subsubsec:laxcomma} presents the lax comma category,
which uses two 2-functors to produce a new 2-category.
Many of these topics are explained with more depth in the literature of category of category theory,
such as \cite{joyofcats,borceux1,grandis,gray,maclane}.

\subsubsection{Category of Sets}\label{subsubsec:sets}

This section briefly introduces some key functors on the category $\set$ of sets with functions.
The functors found here will be used,
implicitly or explicitly,
throughout the entirety of this paper.

The \emph{diagonal functor} $\xymatrix{\set\ar[r]^(0.4){\Delta} & \set\times\set}$ sends a set $X$ to the pair $(X,X)$ \cite[p.\ 62]{maclane}.
From \cite[p.\ 87]{maclane},
$\Delta$ admits a right adjoint functor $\xymatrix{\set\times\set\ar[r]^(0.6){\Delta^{\star}} & \set}$,
which sends a pair of sets $(X,Y)$ to their cartesian product $X\times Y$.
The composition $\xymatrix{\set\ar[r]^{\Delta^{\star}\Delta} & \set}$ yields the \emph{squaring functor},
which sends a set $X$ to its cartesian product with itself \cite[Example 5.38.1]{joyofcats}.

The \emph{covariant power-set functor} $\xymatrix{\set\ar[r]^{\mathcal{P}} & \set}$ sends a set $X$ to its power set.
If $\xymatrix{X\ar[r]^{f} & Y}\in\set$,
then $\xymatrix{\mathcal{P}(X)\ar[r]^{\mathcal{P}(f)} & \mathcal{P}(Y)}\in\set$ has the action
\[
\mathcal{P}(f)(A):=\left\{
f(x)
:
x\in A
\right\},
\]
sending $A$ to its image under $f$ \cite[p.\ 13]{maclane}.

\subsubsection{Category of Quasi-ordered Sets}\label{subsubsec:quasiordered}

This section briefly introduces some key functors on the category $\qord$ of quasi-ordered sets with monotone maps.
Please note that $\qord$ is a 2-category,
much like its subcategory of partially-ordered sets \cite[p.\ 277]{grandis}.
The 0-cells are quasi-ordered sets,
and the 1-cells are monotone maps.
For $\xymatrix{P\ar@/^/[r]^{\phi}\ar@/_/[r]_{\psi} & Q}\in\qord$,
there is at most one 2-cell from $\phi$ to $\psi$ determined by the pointwise ordering:
$\phi\leq\psi$ if and only if $\phi(x)\leq_{Q}\psi(x)$ for all $x\in P$.

The \emph{underlying-set functor} $\xymatrix{\qord\ar[r]^{F} & \set}$,
which strips away the order structure,
admits a left adjoint $\xymatrix{\set\ar[r]^{F^{\diamond}} & \qord}$ given by the trivial ordering \cite[Proposition 3]{benini2019}.
One can regard $F^{\diamond}$ as a 2-functor by imbuing $\set$ with trivial 2-cells.

Another key functor from $\set$ to $\qord$ is the covariant power-set functor.
Indeed,
for any set $X$,
one can equip $\mathcal{P}(X)$ with a natural ordering of subsets,
which will be denoted $\mathcal{P}^{+}(X)$.
Likewise,
for any function $\xymatrix{X\ar[r]^{f} & Y}\in\set$,
an exercise shows that the function $\mathcal{P}(f)$ from $\mathcal{P}^{+}(X)$ to $\mathcal{P}^{+}(Y)$ is monotone.
The notation $\mathcal{P}^{+}(f)$ will denote when $\mathcal{P}(f)$ is considered as a monotone map.
Thus,
$\xymatrix{\set\ar[r]^{\mathcal{P}^{+}} & \qord}$ is a 2-functor like $F^{\diamond}$.

The last functor of interest is the \emph{dual-quasi-order autofunctor}.
For a quasi-ordered set $P$,
the \emph{dual quasi-order} is defined by $x\leq_{P^{\mathrm{T}}} y$ if $y\leq_{P} x$ \cite[p.\ 33]{eklund}.
Let $P^{\mathrm{T}}$ be the set $P$ equipped with the dual quasi-order.
For any monotone map $\xymatrix{P\ar[r]^{\phi} & Q}\in\qord$,
an exercise shows that $\phi$ is monotone from $P^{\mathrm{T}}$ to $Q^{\mathrm{T}}$.
The notation $\phi^{\mathrm{T}}$ will denote when $\phi$ is considered with the dual quasi-orders.
Thus,
$\xymatrix{\qord\ar[r]^{\Box^{\mathrm{T}}} & \qord}$ is a self-inverting 2-functor.

\subsubsection{Lax Comma Category}\label{subsubsec:laxcomma}

This section summarizes the construction of a lax comma category from a pair of 2-functors with a common codomain.
Much like the classical comma category \cite[Definition 1.6.1]{borceux1},
the lax comma category produces a 2-category,
which blends the three 2-categories involved in the 2-functors used in its definition.
The definition given below is modified from \cite[p.\ 145]{gray1980},
in that the 2-cells are reversed.
This choice is made to remove the categorical dualities present in \cite[p.\ 29-30]{gray} and better align with the natural ordering of the real line in Examples \ref{ex:weighted} and \ref{ex:metric}.

\begin{defn}[Lax comma category]
Let $\xymatrix{\cat{A}\ar[r]^{F} & \cat{C} & \cat{B}\ar[l]_{G}}$ be 2-functors.
Define the 2-category $\left(F\backslash\backslash G\right)$ in the following way:
\begin{itemize}

\item a 0-cell is a triple $(A,f,B)$,
where $\xymatrix{F(A)\ar[r]^{f} & G(B)}$ is a 1-cell in $\cat{C}$;

\item a 1-cell from $(A,f,B)$ to $\left(A',f',B'\right)$ is a triple $(x,\alpha,y)$,
where $\xymatrix{A\ar[r]^{x} & A'}$ is a 1-cell in $\cat{A}$,
$\xymatrix{B\ar[r]^{y} & B'}$ is a 1-cell in $\cat{B}$,
and $\xymatrix{f'\circ F(x)\ar[r]^{\alpha} & G(y)\circ f}$ is a 2-cell in $\cat{C}$;
\[\xymatrix{
F(A)\ar[d]_{f}\ar[r]^{F(x)}    &   F\left(A'\right)\ar[d]^{f'}\ar@{=>}[dl]|-{\alpha}\\
G(B)\ar[r]_{G(y)}    &   G\left(B'\right)
}\]

\item composition of 1-cells
\[\xymatrix{
(A,f,B)
\ar[rr]^{(x,\alpha,y)}
&
&
\left(A',f',B'\right)
\ar[rr]^{\left(x',\alpha',y'\right)}
&
&
\left(A'',f'',B''\right)
}\]
is given by
\[
\left(x',\alpha',y'\right)\circ(x,\alpha,y):=\left(
x'\circ x,
\left(id_{G\left(y'\right)}\ast\alpha\right)\odot\left(\alpha'\ast id_{F(x)}\right),
y'\circ y\right);
\]

\item a 2-cell from $(x,\alpha,y)$ to $\left(z,\beta,w\right)$ is a pair $(\sigma,\tau)$,
where $\xymatrix{x\ar[r]^{\sigma} & z}$ is a 2-cell in $\cat{A}$
and $\xymatrix{y\ar[r]^{\tau} & w}$ is a 2-cell in $\cat{B}$ such that
$\beta\odot\left(id_{f'}\ast F(\sigma)\right)=\left(G(\tau)\ast id_{f}\right)\odot\alpha$;
\[\xymatrix{
f'\circ F(x)\ar[d]_{\alpha}\ar[rr]^{id_{f'}\ast F(\sigma)}  &   &
f'\circ F(z)\ar[d]^{\beta}\ar@{}[dll]|-{=}\\
G(y)\circ f\ar[rr]_{G(\tau)\ast id_{f}} &   &
G(w)\circ f
}\]

\item horizontal composition, vertical composition, and identities for 2-cells are componentwise.

\end{itemize}
\end{defn}

The benefit of using lax comma categories over classical comma categories is,
as its name implies,
to relax the need for a commutative square.
Instead,
a 2-cell stands in place of commutativity.
The examples below illustrate how this relaxation manifests in weighted sets and metric spaces,
passing from isometric maps to contractive maps.

\begin{ex}[Weighted sets]\label{ex:weighted}
Let $\cat{1}$ the the discrete 2-category of a single object $1$,
and define a 2-functor $\xymatrix{\cat{1}\ar[r]^(0.4){K} & \qord}$ by $1\mapsto [0,\infty)$,
equipped with its usual ordering.
Then,
the lax comma category $\left(F^{\diamond}\backslash\backslash K\right)$ is isomorphic to the category $\wsetc$ of weighted sets with contractive maps \cite[Definition 2.1.1]{grilliette2015}.
On the other hand,
the traditional comma category $\left(F^{\diamond}\downarrow K\right)$ is isomorphic to the category of weighted sets with isometric maps.

Define a 2-functor $\xymatrix{\cat{1}\ar[r]^(0.4){K_{\infty}} & \qord}$ by $1\mapsto [0,\infty]$,
equipped with its usual ordering.
Then,
the lax comma category $\left(F^{\diamond}\backslash\backslash K_{\infty}\right)$ is isomorphic to the category $\wsetg$ of weighted sets with contractive maps,
where $\infty$ is allowed as a weight value \cite[p.\ 175-176]{grandis2007}.
\end{ex}

\begin{ex}[Metric spaces]\label{ex:metric}
The lax comma category $\left(F^{\diamond}\Delta^{\star}\Delta\backslash\backslash K_{\infty}\right)$ is isomorphic to the category $\schc$ defined in the following way:
\begin{itemize}
\item an object of $\schc$ is a \emph{schwach metric space},
a pair $\left(X,d_{X}\right)$ where $d_{X}:X\times X\to[0,\infty]$;
\item a morphism $\xymatrix{\left(X,d_{X}\right)\ar[r]^{f} & \left(X',d_{X'}\right)}\in\schc$ is a function $f:X\to X'$ such that $d_{X'}\left(f(x),f(y)\right)\leq d_{X}(x,y)$ for all $x,y\in X$.
\end{itemize}
Both the category $\metg$ of directed metric spaces with d-contractions \cite[p.\ 119]{grandis2004}
and the category $\metc$ of metric spaces with contractive maps \cite[Example 3.3.3.a]{joyofcats} are full subcategories of $\schc$.
On the other hand,
the traditional comma category $\left(F^{\diamond}\Delta^{\star}\Delta\downarrow K_{\infty}\right)$ is isomorphic to the category of schwach metric spaces with functions satisfying
$d_{X'}\left(f(x),f(y)\right)=d_{X}(x,y)$
for all $x\in X$,
which are precisely isometric maps.
\end{ex}

\section{The Fallacy of Intuition}\label{sec:fallacy}

This section addresses some concerns regarding traditional approaches to graph theory.
First,
the literature is inconsistent with the use of the term ``graph'',
where the word can take multiple different meanings:
a set with a family of unordered pairs (a simple graph),
a set with a family of ordered pairs (a digraph),
a pair of sets with a function from one set to the square of the other set (a quiver),
a pair of sets with a function from one set to the power set of the other set (a hypergraph).
While these meanings are interrelated,
they are by no means the interchangeable,
which can be vividly illustrated by considering the category of each class of object.
Together,
these different categories and the transformations between them form the diagram in Figure \ref{fig:diagram},
which pictorially represents how the structures interact.

Yet,
Figure \ref{fig:diagram} leads to the second,
and more worrisome,
concern.
Several functors in the diagram are mirrors of each other,
performing a similar operation but between different structures.
For example,
the deletion from set-system hypergraphs to multigraphs acts much the same as the deletion form set systems to graphs.
However,
while some functors mirror each other,
they do not act the same way.
Specifically,
the associated digraph functor mirrors the inclusion of symmetric digraphs into all digraphs,
but they do not and cannot have the same behavior.
Sadly,
the intuition of similar processes having similar behavior does not hold here.

The failure of intuition extends to the clique-replacement graph,
the intersection graph,
and the dual hypergraph.
In \cite[p.\ 187-188]{dorfler1980},
all three operations were reported to be functors with the details left to the reader.
Unfortunately,
none of these operations are functorial,
which can be shown with simple examples.
Consequently,
care must be taken when applying such intuition,
regardless of how obvious the intuition might seem to be.

Section \ref{subsec:classical} gathers together seven categories of graphs,
only two of which are legitimately isomorphic,
and their natural relationships to each other.
The section culminates in Figure \ref{fig:diagram},
which visually depicts these connections.
Section \ref{subsec:cliqueintersection} briefly considers clique-replacement and intersection graphs,
and highlights their failure to be functorial.
Section \ref{subsec:dualhypergraph} also addresses the dual hypergraph and,
like the clique-replacement and intersection graphs,
shows the construction not to be functorial.

\subsection{A Diagram of Classical Graph Theory}\label{subsec:classical}

This section pulls together natural operations between different types of graphs,
culminating in Figure \ref{fig:diagram}.
Notably,
while several adjunctions are present in Figure \ref{fig:diagram},
there is only one legitimate isomorphism or equivalence.
These functors will be utilized later in this paper,
and this section will establish notation for each operator.

Yet, these functors are used here to demonstrate a particular failure of intuition.
Specifically,
for simple graphs,
an undirected graph can be viewed as a symmetric digraph,
with all the categorical structure intact.
Unfortunately,
the same cannot be said for set-system multigraphs,
which cannot be embedded into the category of quivers.
Therefore,
care must be used when attempting to lift results from one type of graph to another,
no matter how intuitive they might seem.

Section \ref{subsubsec:digraphs} describes the connection between quivers and digraphs.
Section \ref{subsubsec:ssystems} details the relationship between set-system hypergraphs and set systems.
Section \ref{subsubsec:graphs} lastly considers the precarious placement of set-system multigraphs and graphs in the larger scheme of graph-theoretic objects.
In particular,
the category of graphs is isomorphic to the category of symmetric digraphs,
which is both reflective and coreflective in the larger category of all digraphs.
On the other hand,
the category of graphs is only coreflective in the larger category of all set systems,
showing that simple graphs seem more closely related to digraphs than set systems.

\subsubsection{Quivers \& Digraphs}\label{subsubsec:digraphs}

This section considers the relationship between directed graphs and directed multigraphs.
Directed multigraphs,
or quivers,
arise naturally as a category of presheaves $\cat{Q}:=\set^{\cat{E}}$,
where $\cat{E}$ is the finite category drawn below.
\[\xymatrix{
1\ar@/^/[rr]^{s}\ar@/_/[rr]_{t}   &   &   0
}\]
An object $Q=\left(\vec{V}(Q),\vec{E}(Q),\sigma_{Q},\tau_{Q}\right)$ of $\cat{Q}$ consists of a vertex set $\vec{V}(Q)$,
an edge set $\vec{E}(Q)$,
a source map $\sigma_{Q}$,
and a target map $\tau_{Q}$,
which agrees with \cite{joyofcats,bumby1984,raeburn,schiffler}.

On the other hand,
the category of digraphs,
or relations,
is more easily understood as a functor-structured category $\digra:=\spa\left(\Delta^{\star}\Delta\right)$.
An object $Q=\left(\vec{V}(Q),\vec{E}(Q)\right)$ of $\digra$ consists of a vertex set $\vec{V}(Q)$ and a set of ordered pairs $\vec{E}(Q)\subseteq\vec{V}(Q)\times\vec{V}(Q)$,
which agrees with \cite{bang,godsil,hell}.
There is a natural embedding $\xymatrix{\digra\ar[r]^(0.6){\mathsf{N}_{\cat{Q}}} & \cat{Q}}$ with the following action:
\begin{itemize}

\item $\mathsf{N}_{\cat{Q}}(Q):=\left(\vec{V}(Q),\vec{E}(Q),\sigma_{Q},\tau_{Q}\right)$,
where $\sigma_{Q}(v,w):=v$ and $\tau_{Q}(v,w):=w$;

\item $\mathsf{N}_{\cat{Q}}(f):=\left(f,\vec{E}\mathsf{N}_{\cat{Q}}(f)\right)$,
where $\vec{E}\mathsf{N}_{\cat{Q}}(f)(v,w):=\left(f(v),f(w)\right)$.

\end{itemize}
As seen in \cite[p.\ 10]{simplification1},
$\mathsf{N}_{\cat{Q}}$ admits a left adjoint $\xymatrix{\cat{Q}\ar[r]^(0.4){\mathsf{S}_{\cat{Q}}} & \digra}$ with the action below on objects.
\[
\mathsf{S}_{\cat{Q}}(Q)
=\left(
\vec{V}(Q),
\left\{\left(\sigma_{Q}(e),\tau_{Q}(e)\right):e\in\vec{E}(Q)\right\}
\right)
\]

\subsubsection{Set Systems \& Hypergraphs}\label{subsubsec:ssystems}

This section considers the relationship between set-system hypergraphs and set systems.
Set-system hypergraphs naturally result from a comma category $\cat{H}:=\left(id_{\set}\downarrow\mathcal{P}\right)$.
An object $G=\left(E(G),\epsilon_{G},V(G)\right)$ of $\cat{H}$ consists of a vertex set $V(G)$,
an edge set $E(G)$,
and an incidence function $\epsilon_{G}$,
which agrees with \cite{dorfler1980,watkins1973}.

On the other hand,
the category of set systems is quickly built using a functor-structured category $\ssys:=\spa(\mathcal{P})$.
An object $G=\left(V(G),E(G)\right)$ of $\ssys$ consists of a vertex set $V(G)$ and a family of subsets $E(G)\subseteq\mathcal{P}V(G)$,
which agrees with \cite{bondy,hajiabolhassen2016,hammack2016}.
There is a natural embedding $\xymatrix{\ssys\ar[r]^{\mathsf{N}_{\cat{H}}} & \cat{H}}$ with the following action:
\begin{itemize}

\item $\mathsf{N}_{\cat{H}}(G):=\left(E(G),\epsilon_{G},V(G)\right)$,
where $\epsilon_{G}(A):=A$;

\item $\mathsf{N}_{\cat{H}}(f):=\left(E\mathsf{N}_{\cat{H}}(f),f\right)$,
where $E\mathsf{N}_{\cat{H}}(f)(A):=\mathcal{P}(f)(A)$.

\end{itemize}
As seen in \cite[p.\ 9]{simplification1},
$\mathsf{N}_{\cat{H}}$ admits a left adjoint $\xymatrix{\cat{H}\ar[r]^(0.4){\mathsf{S}_{\cat{H}}} & \ssys}$ with the action below on objects.
\[
\mathsf{S}_{\cat{H}}(G)
=\left(
V(G),
\left\{\epsilon_{G}(e):e\in E(G)\right\}
\right)
\]

\subsubsection{Multigraphs \& Simple Graphs}\label{subsubsec:graphs}

At last,
this section considers multigraphs,
simple graphs,
and their precarious place seated between quivers and set-system hypergraphs.
Indeed,
the functors of this section blend with the simplification functors of the previous two sections to produce Figure \ref{fig:diagram},
which captures various perspectives of graph theory and their interactions.

Let $\cat{M}$ be the full subcategory of $\cat{H}$ consisting of multigraphs,
i.e.\ set-system hypergraphs $G$ such that each edge has at least one endpoint and no more than two endpoints,
which agrees with \cite{bondy,dorfler1980}.
The inclusion $\xymatrix{\cat{M}\ar[r]^{N_{\cat{H}}} & \cat{H}}$ admits a right adjoint $\del$ given on objects by the deletion below \cite[Theorem 2.3.3]{grilliette2023}.
\[
\del(G)
=\left(
\left\{e\in E(G):1\leq\card\left(\epsilon_{G}(e)\right)\leq2\right\},
\left.\epsilon_{G}\right|^{E\del(G)},
V(G)
\right)
\]
A quiver can be transformed into a multigraph by stripping away the direction,
which is encoded into the underlying multigraph functor $\xymatrix{\cat{Q}\ar[r]^{U} & \cat{M}}$ with the following action:
\begin{itemize}
\item $U(Q):=\left(\vec{E}(Q),\epsilon_{U(Q)},\vec{V}(Q)\right)$,
where $\epsilon_{U(Q)}(e):=\left\{\sigma_{Q}(e),\tau_{Q}(e)\right\}$;
\item $U(\phi):=\left(\vec{E}(\phi),\vec{V}(\phi)\right)$.
\end{itemize}
By \cite[Theorem 2.3.7]{grilliette2023},
$U$ admits a right adjoint $\vec{D}$ defined on objects in the following way:
\[
\vec{D}(G)
=\left(
V(G),
\left\{(e,v,w):\epsilon_{G}(e)=\{v,w\}\right\},
\sigma_{\vec{D}(G)},
\tau_{\vec{D}(G)}
\right),
\]
where $\sigma_{\vec{D}(G)}(e,v,w):=v$ and $\tau_{\vec{D}(G)}(e,v,w):=w$.

On the other hand,
let $\gra$ be the full subcategory of $\ssys$ consisting of all conventional graphs,
i.e.\ set systems such that each edge has at least one endpoint and no more than two endpoints,
which agrees with \cite{bang,godsil,hell,kilp2001}.
Mimicking the multigraph result,
the inclusion $\xymatrix{\gra\ar[r]^{N_{\ssys}} & \ssys}$ admits a right adjoint functor in much the same way,
deleting all nontraditional edges.
The proof of the characterization is a simplified version of its multigraph counterpart.

\begin{defn}[Simple deletion]
Given a set system $H$,
define a graph $\del_{\mathcal{S}}(H):=\left(V(H),E\del_{\mathcal{S}}(H)\right)$,
where
\[
E\del_{\mathcal{S}}(H):=\left\{A\in E(H):1\leq\card(A)\leq 2\right\}.
\]
Let $\xymatrix{\del_{\mathcal{S}}(H)\ar[r]^(0.65){j_{H}} & H}\in\ssys$ be the canonical inclusion homomorphism from $\del_{\mathcal{S}}(H)$ into $H$.
\end{defn}

\begin{thm}[Characterization of $\del_{\mathcal{S}}$]
If $\xymatrix{N_{\ssys}\left(H'\right)\ar[r]^(0.65){f} & H}\in\ssys$,
there is a unique $\xymatrix{H'\ar[r]^(0.4){\hat{f}} & \del_{\mathcal{S}}(H)}\in\gra$ such that $j_{H}\circ N_{\ssys}\left(\hat{f}\right)=f$.
\end{thm}

Much like \cite[Theorem 2.3.4]{grilliette2023},
$N_{\ssys}$ fails to preserve products.
Consequently,
$\gra$ is a coreflective subcategory of $\ssys$,
but not reflective.

One can now conjugate the simplification of $\cat{H}$ by the deletion adjunctions to spawn a simplification for $\cat{M}$.
Direct calculation shows that this simplification of a multigraph is the traditional simplification to the simple graph.
The proof is identical to a direct proof of the hypergraph and quiver cases in \cite[p.\ 9-10]{simplification1}.

\begin{defn}[Multigraph simplification]
Define the composite functors
\[\begin{array}{ccc}
\mathsf{N}_{\cat{M}}:=\del\mathsf{N}_{\cat{H}}N_{\ssys}
&
\textrm{and}
&
\mathsf{S}_{\cat{M}}:=\del_{\mathcal{S}}\mathsf{S}_{\cat{H}}N_{\cat{H}}.
\end{array}
\]
For a multigraph $G$,
define $\xymatrix{G\ar[r]^(0.4){\mu_{G}} & \mathsf{N}_{\cat{M}}\mathsf{S}_{\cat{M}}(G)}\in\cat{M}$ by
\begin{itemize}
\item $V\left(\mu_{G}\right)(v):=v$,
\item $E\left(\mu_{G}\right)(e):=\epsilon_{G}(e)$.
\end{itemize}
\end{defn}

\begin{thm}[Characterization of $\mathsf{S}_{\cat{M}}$]
If $\xymatrix{G\ar[r]^(0.4){\phi} & \mathsf{N}_{\cat{M}}\left(G'\right)}\in\cat{M}$,
there is a unique $\xymatrix{\mathsf{S}_{\cat{M}}\left(G\right)\ar[r]^(0.6){\hat{\phi}} & G'}\in\gra$ such that $\mathsf{N}_{\cat{M}}\left(\hat{\phi}\right)\circ\mu_{G}=\phi$.
\end{thm}

Having adapted simplification and deletion from $\cat{H}$ to $\cat{M}$,
a next logical step would be to adapt the underlying multigraph functor from $\cat{Q}$ to $\digra$.
However,
the transition from digraphs to graphs has more subtlety than its quiver counterpart.

To facilitate the exposition,
let $\sdigra$ be the full subcategory of $\digra$ consisting of all symmetric digraphs,
i.e.\ digraphs $Q$ such that $(v,w)\in\vec{E}(Q)$ implies $(w,v)\in\vec{E}(Q)$,
which agrees with \cite{dochtermann2009,dochtermann2009-2,hell}.
Like the associated digraph functor $\vec{D}$,
the inclusion functor $\xymatrix{\sdigra\ar[rr]^{N_{\digra}} & & \digra}$ admits a left adjoint given by the symmetric closure.
The proof of the universal property is routine.

\begin{defn}[{Symmetric closure, \cite[Theorem 4.5.5]{howtoproveit}}]
Given a digraph $Q$,
define the symmetric digraph $N^{\diamond}_{\digra}(Q):=\left(\vec{V}(Q),\vec{E}N_{\digra}^{\diamond}(Q)\right)$,
where
\[
\vec{E}N_{\digra}^{\diamond}(Q):=\left\{(v,w):(v,w)\in\vec{E}(Q)\textrm{ or }(w,v)\in\vec{E}(Q)\right\}.
\]
Let $\xymatrix{Q\ar[r]^(0.3){\kappa_{Q}} & N_{\digra}N^{\diamond}_{\digra}(Q)}\in\digra$ be the canonical inclusion homomorphism of $Q$ into $N^{\diamond}_{\digra}(Q)$.
\end{defn}

\begin{thm}[Characterization of $N^{\diamond}_{\digra}$]
If $\xymatrix{Q\ar[r]^(0.3){\phi} & N_{\digra}\left(Q'\right)}\in\digra$,
there is a unique $\xymatrix{N^{\diamond}_{\digra}(Q)\ar[r]^(0.65){\hat{\phi}} & Q'}\in\sdigra$ such that $N_{\digra}\left(\hat{\phi}\right)\circ\kappa_{Q}=\phi$.
\end{thm}

Also,
$N_{\digra}$ admits a right adjoint given by the symmetric interior,
and its universal property is proven in a similar manner.

\begin{defn}[{Symmetric interior, \cite[Lemma 3.6]{neuzerling2016}}]
Given a digraph $Q$,
define the symmetric digraph $N^{\star}_{\digra}(Q):=\left(\vec{V}(Q),\vec{E}N_{\digra}^{\star}(Q)\right)$,
where
\[
\vec{E}N_{\digra}^{\star}(Q):=\left\{(v,w):(v,w)\in\vec{E}(Q)\textrm{ and }(w,v)\in\vec{E}(Q)\right\}.
\]
Let $\xymatrix{N_{\digra}N^{\star}_{\digra}(Q)\ar[r]^(0.7){\nu_{Q}} & Q}\in\digra$ be the canonical inclusion homomorphism of $N^{\star}_{\digra}(Q)$ into $Q$.
\end{defn}

\begin{thm}[Characterization of $N^{\star}_{\digra}$]
If $\xymatrix{N_{\digra}\left(Q'\right)\ar[r]^(0.7){\phi} & Q}\in\digra$,
there is a unique $\xymatrix{Q'\ar[r]^(0.3){\hat{\phi}} & N_{\digra}^{\star}\left(Q\right)}\in\sdigra$ such that $\nu_{Q}\circ N_{\digra}\left(\hat{\phi}\right)=\phi$.
\end{thm}

Consequently,
$\sdigra$ is both a reflective and coreflective subcategory of $\digra$,
inheriting much of the latter's structure without modification,
much of which was inherited from $\cat{Q}$ \cite[Theorem 4.3]{simplification1}.

Yet,
$\sdigra$ and $\gra$ are isomorphic as categories through the equivalent symmetric digraph construction $\xymatrix{\gra\ar[r]^(0.4){Z_{\gra}} & \sdigra}$ with the action below \cite[p.\ 21]{rbl1}.
\begin{itemize}

\item $Z_{\gra}(G):=\left(
V(G),
\left\{(v,w):\{v,w\}\in E(G)\right\}
\right)$

\item $Z_{\gra}(f):=f$

\end{itemize}
This construction is ubiquitous to the point that the distinction between a graph and a symmetric digraph is effectively lost.
Considering $Z_{\gra}$ is an isomorphism,
such a blurring is not unwarranted,
but unfortunately,
doing so can lead to fallacious reasoning when trying to lift to the analogous connection between $\cat{Q}$ and $\cat{M}$.

\begin{figure}
\[\xymatrix{
&   &
\cat{M}\ar@/_/[drr]\ar@/^/[dll]^{\vec{D}}_{\dashv}\ar@/_/[dddr]_{\mathsf{S}_{\cat{M}}}^{\dashv}\\
\cat{Q}\ar@/_/[d]_{\mathsf{S}_{\cat{Q}}}^{\dashv}\ar@/^/[urr]^{U} &   &   &   &
\cat{H}\ar@/_/[d]_{\mathsf{S}_{\cat{H}}}^{\dashv}\ar@/_/[ull]_{\del}^{\dashv}\\
\digra\ar@/_/[u]_{\mathsf{N}_{\cat{Q}}}\ar@/^1pc/[dr]^{\star}_{\dashv}\ar@/_1pc/[dr]_{\diamond}^{\dashv} &   &    &   &
\ssys\ar@/_/[u]_{\mathsf{N}_{\cat{H}}}\ar@/^/[dl]^{\del_{\mathcal{S}}}_{\dashv}\\
&
\sdigra\ar[ul]\ar@{<->}[rr]_{\cong} &   &
\gra\ar@/^/[ur]\ar@/_/[uuul]_{\mathsf{N}_{\cat{M}}}
}\]
\caption{Diagram of Classical Graph Theory}
\label{fig:diagram}
\end{figure}
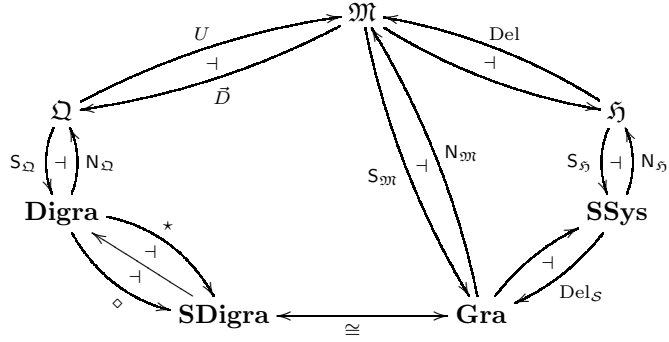

Indeed,
consider the diagram in Figure \ref{fig:diagram}.
Direct computation demonstrates the following functorial equalities,
which not only shows that simplification and deletion commute with the inclusions,
but the connection between the underlying multigraph adjunction and the inclusion $N_{\sdigra}$.

\begin{thm}[Compatibility]\label{thm:compatibility}
The following functorial equalities hold:
\begin{itemize}
\item $N_{\ssys}\mathsf{S}_{\cat{M}}=\mathsf{S}_{\cat{H}}N_{\cat{H}}$,
\item $\mathsf{N}_{\cat{H}}N_{\ssys}=N_{\cat{H}}\mathsf{N}_{\cat{M}}$,
\item $Z_{\gra}\mathsf{S}_{\cat{M}}U=N_{\digra}^{\diamond}\mathsf{S}_{\cat{Q}}$,
\item $N_{\digra}Z_{\gra}\mathsf{S}_{\cat{M}}=\mathsf{S}_{\cat{Q}}\vec{D}$.
\end{itemize}
Consequently, the following isomorphism is natural:
\begin{itemize}
\item $\mathsf{N}_{\cat{M}}\del_{\mathcal{S}}(H)\cong\del\mathsf{N}_{\cat{H}}(H)$ for a set system $H$,
\item $\vec{D}\mathsf{N}_{\cat{M}}(G)\cong\mathsf{N}_{\cat{Q}}N_{\digra}Z_{\gra}(G)$ for a graph $G$.
\end{itemize}
\end{thm}

Notice that through 
the simplification adjunction and the isomorphism $Z_{\gra}$,
the underlying multigraph functor $U$ corresponds to the symmetric closure $N_{\digra}^{\diamond}$,
and the associated digraph functor $\vec{D}$ corresponds to the inclusion $N_{\digra}$.
In fact,
$\vec{D}$ is intertwined with $N_{\digra}$ through both the simplifications and inclusions,
i.e.\ both left and right adjoints.

But,
$\vec{D}$ does not share the same structure as $N_{\digra}$.
If $\vec{D}$ was equivalent to an inclusion of a full subcategory of $\cat{Q}$,
which was both reflective and coreflective,
then $\vec{D}$ would admit not only a left adjoint $U$,
but also a right adjoint.
Sadly,
$\vec{D}$ is not cocontinuous \cite[Lemma 2.3.8]{grilliette2023},
which precludes such a right adjoint.
Said another way,
there is no analogue to the symmetric interior for quivers,
and the lemma demonstrates why.
The structure of quotients in $\cat{Q}$ is an obstruction.

Therefore,
despite superficial similarities,
one should be careful when applying such intuition.
How structures are represented,
and how they are mapped,
play a pivotal role in how they behave.

\subsection{Clique Replacement \& Intersection Graphs}\label{subsec:cliqueintersection}

This section discusses two graph operations,
which are prevalent in the literature
but are not represented in Figure \ref{fig:diagram}.
Both are methods of converting a hypergraph to a simple graph,
which have been studied previously.
First,
the clique-replacement graph takes a set-system hypergraph and replaces each edge with a clique,
taking every pairwise relationship found within the edge.
Notably,
the self-relationship of a vertex with itself inside an edge is omitted.

\begin{defn}[{Clique-replacement graph, \cite[p.\ 188]{dorfler1980}}]
Given a set-system hypergraph $G$,
define the graph $\Gamma(G)$ by
\begin{itemize}
\item $V\Gamma(G):=V(G)$,
\item $E\Gamma(G):=\left\{\{v,w\}:\exists e\in E(G)\left(\{v,w\}\subseteq\epsilon_{G}(e)\right),v\neq w\right\}$.
\end{itemize}
\end{defn}

Likewise,
the intersection graph
transforms a set-system hypergraph by converting the edge set into the vertex set,
and declaring two edges adjacent if their endpoint sets overlap.
If performed on a multigraph,
this process is precisely the line graph.
Again,
note that self-adjacency due to an edge's endpoint set intersecting itself is omitted.

\begin{defn}[{Intersection graph, \cite[p.\ 188]{dorfler1980}}]
Given a set-system hypergraph $G$,
define the graph $L(G)$ by
\begin{itemize}
\item $VL(G):=E(G)$,
\item $EL(G):=\left\{\{e,f\}:\epsilon_{G}(e)\cap\epsilon_{G}(f)\neq\emptyset,e\neq f\right\}$.
\end{itemize}
\end{defn}

Unfortunately,
as written,
neither of these operations can be extended to a functor from $\cat{H}$ to $\gra$.
Contrary to \cite[p.\ 188]{dorfler1980},
simple counterexamples demonstrate that some adjacencies fail to be preserved by each construction,
namely those adjacencies that would force the very self-adjacencies omitted above.
Much like the dissimilarity between $\vec{D}$ and $N_{\digra}$,
rigor must be employed to temper intuition.

\begin{prop}[{Failure of \cite[Proposition 3.3]{dorfler1980}}]
No functor $\xymatrix{\cat{H}\ar[r]^(0.4){M} & \gra}$ satisfies that $M(G)=\Gamma(G)$ for all set-system hypergraphs $G$.
\end{prop}

\begin{proof}

A pair of hypergraphs $G$ and $H$ will be demonstrated such that the homomorphism set $\cat{H}(G,H)\neq\emptyset$,
but the homomorphism set $\gra\left(\Gamma(G),\Gamma(H)\right)=\emptyset$.
Consequently,
if such a functor $M$ existed,
it would map a nonempty set of homomorphisms to an empty set of homomorphisms,
which is absurd.

Consider the set-system hypergraphs $G$ and $H$ drawn below.
\begin{center}
\includegraphics[scale=1]{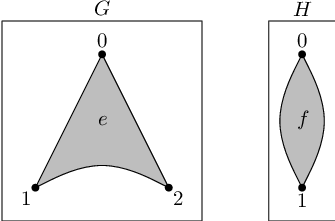}
\end{center}
Define $\xymatrix{G\ar[r]^{\phi} & H}\in\cat{H}$ by $V(\phi)(n):=n\mod 2$ and $E(\phi)(e):=f$.
Applying $\Gamma$ produces the two graphs below.
\begin{center}
\includegraphics[scale=1]{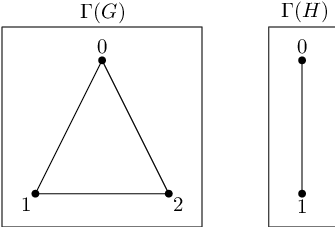}
\end{center}
Say $\xymatrix{\Gamma(G)\ar[r]^{f} & \Gamma(H)}\in\gra$.
Then,
\begin{center}$
\{f(0),f(1)\}
=\{f(1),f(2)\}
=\{f(0),f(2)\}
=\{0,1\}.
$\end{center}
By the Pigeonhole Principle,
there are $v,w\in V(G)$ such that $f(v)=f(w)$ and $v\neq w$.
Without loss of generality,
say $v=0$ and $w=1$.  Then,
\begin{center}$
\{0,1\}
=\{f(0),f(1)\}
=\{f(v),f(w)\}
=\{f(v)\},
$\end{center}
which is absurd.

\end{proof}

\begin{prop}[{Failure of \cite[Proposition 3.4]{dorfler1980}}]
No functor $\xymatrix{\cat{H}\ar[r]^(0.4){M} & \gra}$ satisfies that $M(G)=L(G)$ for all set-system hypergraphs $G$.
\end{prop}

\begin{proof}

A pair of hypergraphs $G$ and $H$ will be demonstrated such that the homomorphism set $\cat{H}(G,H)\neq\emptyset$,
but the homomorphism set $\gra\left(L(G),L(H)\right)=\emptyset$.
Consequently,
if such a functor $M$ existed,
it would map a nonempty set of homomorphisms to an empty set of homomorphisms,
which is absurd.

Consider the set-system hypergraphs $G$ and $H$ drawn below.
\begin{center}
\includegraphics[scale=1]{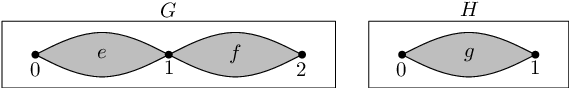}
\end{center}
Define $\xymatrix{G\ar[r]^{\phi} & H}\in\cat{H}$ by $V(\phi)(n):=n\mod 2$ and $E(\phi)(x):=g$.
Applying $L$ produces the two graphs below.
\begin{center}
\includegraphics[scale=1]{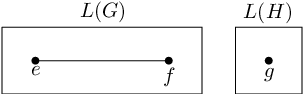}
\end{center}
Say $\xymatrix{L(G)\ar[r]^{h} & L(H)}\in\gra$.
Then,
$h(e)=h(f)=g$,
\[
\{g\}
=\left\{h(e),h(f)\right\}
\in E\left(L(H)\right)
=\emptyset,
\]
which is absurd.

\end{proof}

\subsection{Dual Hypergraph}\label{subsec:dualhypergraph}

This section discusses an operation related to the intersection graph.
The intersection graph converts the edge set of a hypergraph into a new vertex set,
and creates edges to represent how the edges touch one another.
The dual hypergraph pushes this notion further,
completely reversing the roles of the vertex set and edge set,
not unlike the incidence dual of an incidence hypergraph \cite[Lemma 3.1.2]{grilliette2023}.

\begin{defn}[{Dual hypergraph, \cite[p.\ 187]{dorfler1980}}]
Given a set-system hypergraph $G$,
define the set-system hypergraph $d(G)$ by
\begin{itemize}
\item $Vd(G):=E(G)$, $Ed(G):=V(G)$,
\item $\epsilon_{d(G)}(v):=\left\{e:v\in\epsilon_{G}(e)\right\}$.
\end{itemize}
\end{defn}

Sadly,
much like $\Gamma$ and $L$,
this operation cannot be extended to a functor from $\cat{H}$ to itself.
Contrary to \cite[p.\ 187]{dorfler1980},
a simple counterexample demonstrates that some adjacencies fail to be preserved.
While superficially similar to the incidence dual of an incidence hypergraph,
rigor must be employed to temper intuition.

\begin{prop}[{Failure of \cite[Proposition 3.1]{dorfler1980}}]\label{prop:duality}
No functor $\xymatrix{\cat{H}\ar[r]^{M} & \cat{H}}$ satisfies that $M(G)=d(G)$ for all set-system hypergraphs $G$,.
\end{prop}

\begin{proof}

A pair of hypergraphs $G$ and $H$ will be demonstrated such that the homomorphism set $\cat{H}(G,H)\neq\emptyset$,
but the homomorphism set $\cat{H}\left(d(G),d(H)\right)=\emptyset$.
Consequently,
if such a functor $M$ existed,
it would map a nonempty set of homomorphisms to an empty set of homomorphisms,
which is absurd.

Consider the set-system hypergraphs $G$ and $H$ drawn below.
\begin{center}
\includegraphics{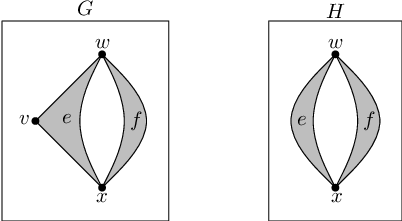}
\end{center}
Define $\xymatrix{G\ar[r]^{\phi} & H}\in\cat{H}$ by
\begin{itemize}
\item $V(\phi)(v):=V(\phi)(w):=w$, $V(\phi)(x):=x$,
\item $E(\phi)(e):=e$, $E(\phi)(f):=f$.
\end{itemize}
Applying $d$ produces the two hypergraphs below.
\begin{center}
\includegraphics{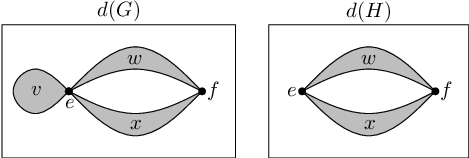}
\end{center}
Say $\xymatrix{d(G)\ar[r]^{\varphi} & d(H)}\in\cat{H}$.
Then,
$E(\varphi)(v)\in Ed(H)=\{w,x\}$.
Without loss of generality,
say $E(\varphi)(v)=w$.
Then,
\[\begin{array}{rcl}
\{e,f\}
&   =   &   \epsilon_{d(H)}\left(w\right)
=\epsilon_{d(H)}\left(E(\varphi)(v)\right)
=\mathcal{P}V(\varphi)\left(\epsilon_{d(G)}(v)\right)
=\mathcal{P}V(\varphi)\left(\{e\}\right)\\
&   =   &   \left\{V(\varphi)(e)\right\},
\end{array}\]
which is absurd.

\end{proof}

\section{An Incidence View of Homomorphisms}\label{sec:incidenceview}

This section presents another perspective on homomorphisms of graphs.
Traditionally,
a graph homomorphism preserves adjacency by the image of an edge in the domain becoming an edge in the codomain.
Unfortunately,
this classical view of homomorphism produces the set-system categories deconstructed in \cite[p.\ 7-20]{grilliette2023},
which possess a litany of categorical issues.

However,
another type of homomorphism exists between hypergraphs,
which weakens the traditional notion of homomorphism in direct analogy to how metric contractions weaken metric isometries.
The category of set-system hypergraphs with these weak homomorphisms can be realized as a lax comma category,
but also is isomorphic to the category of incidence structures.
This revelation results in the diagram in Figure \ref{fig:diagram2},
which demonstrates how traditional graph theory is subsumed by incidence theory.
Analysis of the inclusions,
their adjoints,
and the isomorphism shows that incidence theory cannot be replaced by traditional graph-theoretic methods.
Indeed,
the incidence hypergraph homomorphisms are too distinct from traditional graph homomorphisms for the former to be transformed into the latter without necessarily modifying the hypergraphs themselves.

Section \ref{subsec:incidencehypergraphs} describes the connection between incidence hypergraphs and incidence structures.
Section \ref{subsec:weakhomomorphisms} introduces weak set-system homomorphisms and demonstrates their intrinsic link to maps of incidence structures.
The inclusion of traditional graph homomorphisms to weak set-system homomorphisms admits a right adjoint,
given by a simplicial replacement construction.
Despite this adjunction,
no functor can exist from weak set-system homomorphisms to traditional graph homomorphisms that leaves the hypergraphs themselves unchanged.

\subsection{Incidence Structures \& Hypergraphs}\label{subsec:incidencehypergraphs}

This section considers the relationship between incidence hypergraphs and incidence structures.
Incidence hypergraphs arise naturally as a category of presheaves $\cat{R}:=\set^{\cat{D}}$,
where $\cat{D}$ is the finite category drawn below.
\[\xymatrix{
0   &   &   2\ar[ll]_{y}\ar[rr]^{z}   &   &   1
}\]
An object $G=\left(\check{V}(G),\check{E}(G),I(G),\varsigma_{G},\omega_{G}\right)$ of $\cat{R}$ consists of a vertex set $\check{V}(G)$,
an edge set $\check{E}(G)$,
an incidence set $I(G)$,
a port map $\varsigma_{G}$,
and an attachment map $\omega_{G}$,
which agrees with \cite{chen2018,grilliette2022,grilliette2023,rusnak2018}.

On the other hand,
the category of incidence structures is more easily understood as a functor-structured category $\istr:=\spa\left(\Delta^{\star}\right)$.
An object $G=\left(\check{V}(G),\check{E}(G),I(G)\right)$ of $\istr$ consists of a vertex set $\check{V}(G)$,
an edge set $\check{E}(G)$,
and a set of ordered pairs $I(G)\subseteq\check{V}(G)\times\check{E}(G)$,
which agrees with \cite{beth,bumby1984,dembowski}.
There is a natural embedding $\xymatrix{\istr\ar[r]^(0.6){\mathsf{N}_{\cat{R}}} & \cat{R}}$ with the following action:
\begin{itemize}

\item $\mathsf{N}_{\cat{R}}(G):=\left(\check{V}(G),\check{E}(G),I(G),\varsigma_{G},\omega_{G}\right)$,
where $\varsigma_{G}(v,e):=v$ and $\omega_{G}(v,e):=e$;

\item $\mathsf{N}_{\cat{R}}(f,g):=\left(f,g,I\mathsf{N}_{\cat{R}}(f,g)\right)$,
where $I\mathsf{N}_{\cat{R}}(f,g)(v,e):=\left(f(v),g(e)\right)$.

\end{itemize}
As seen in \cite[p.\ 11]{simplification1},
$\mathsf{N}_{\cat{R}}$ admits a left adjoint $\xymatrix{\cat{R}\ar[r]^(0.4){\mathsf{S}_{\cat{R}}} & \istr}$ with the action below on objects.
\[
\mathsf{S}_{\cat{R}}(G)
=\left(
\check{V}(G),
\check{E}(G),
\left\{\left(\varsigma_{G}(j),\omega_{G}(j)\right):j\in I(G)\right\}
\right)
\]

\subsection{Weak Set-system Homomorphisms}\label{subsec:weakhomomorphisms}

This section introduces a weaker notion of homomorphism for set-system hypergraphs.
Rather than the image of an edge from the domain constituting an edge in the codomain,
the image of an edge from the domain will be contained in an edge in the codomain.
Said notion serves as a generalization of the homomorphisms from \cite[p.\ 84]{bretto} to allow for parallel edges.

\begin{defn}[Weak homomorphism]
Given set-system hypergraphs $G$ and $H$,
a \emph{weak set-system homomorphism} from $G$ to $H$ is a pair $\phi=\left(E^{+}(\phi),V^{+}(\phi)\right)$ satisfying the following conditions:
\begin{itemize}
\item $E^{+}(\phi)$ is a function from $E(G)$ to $E(H)$,
\item $V^{+}(\phi)$ is a function from $V(G)$ to $V(H)$,
\item $
\mathcal{P}\left(V^{+}(\phi)\right)\left(\epsilon_{G}(e)\right)
\subseteq
\epsilon_{H}\left(E^{+}(\phi)(e)\right)
$
for all $e\in E(G)$.
\end{itemize}
\[\xymatrix{
E(G)
\ar[d]_{\epsilon_{G}}
\ar[rr]^{E^{+}(\phi)}
\ar@{}[drr]|-{\subseteq}
&
&
E(H)
\ar[d]^{\epsilon_{H}}\\
\mathcal{P}V(G)
\ar[rr]_{\mathcal{P}\left(V^{+}(\phi)\right)}
&
&
\mathcal{P}V(H)\\
}\]
\end{defn}

If $\xymatrix{G\ar[r]^{\phi} & H}\in\cat{H}$,
then $\phi$ is a weak set-system homomorphism,
which motivates the qualifier ``weak''.
An exercise shows that weak set-system homomorphisms are closed on componentwise composition,
showing that the structure is a supercategory of $\cat{H}$.

\begin{defn}[Category]
Let $\cat{H}^{+}$ be the category of set-system hypergraphs with weak set-system homomorphisms,
and $\xymatrix{\cat{H}\ar[r]^{N_{\cat{H}^{+}}} & \cat{H}^{+}}$ be the inclusion functor.
Also,
let $\xymatrix{\set & \cat{H}^{+}\ar[r]^{V^{+}}\ar[l]_(0.4){E^{+}} & \set}$ be the edge and vertex functors, respectively.
Thus,
$V^{+}N_{\cat{H}^{+}}=V$ and $E^{+}N_{\cat{H}^{+}}=E$.
\end{defn}

Since $\cat{H}$ is not a full subcategory of $\cat{H}^{+}$,
$\cat{H}$ cannot be a reflective or coreflective subcategory $\cat{H}^{+}$.
Despite this,
the inclusion functor $N_{\cat{H}^{+}}$ admits a right adjoint,
which replaces each edge with an abstract simplicial complex.

\begin{defn}[Simplicial replacement]
Given a set-system hypergraph $G$,
define the set-system hypergraph $N_{\cat{H}^{+}}^{\star}(G)$ by
\begin{itemize}
\item $VN_{\cat{H}^{+}}^{\star}(G):=V^{+}(G)$,
\item $EN_{\cat{H}^{+}}^{\star}(G):=\left\{(e,A)\in E^{+}(G)\times\mathcal{P}V^{+}(G):A\subseteq\epsilon_{G}(e)\right\}$,
\item $\epsilon_{N_{\cat{H}^{+}}^{\star}(G)}(e,A):=A$.
\end{itemize}
Define $\xymatrix{N_{\cat{H}^{+}}N_{\cat{H}^{+}}^{\star}(G)\ar[r]^(0.7){\theta_{G}} & G}\in\cat{H}^{+}$ by
\begin{itemize}
\item $V^{+}\left(\theta_{G}\right)(v):=v$,
\item $E^{+}\left(\theta_{G}\right)(e,A):=e$.
\end{itemize}
\end{defn}

\begin{thm}[Universal property of $N_{\cat{H}^{+}}^{\star}$]
Given $\xymatrix{N_{\cat{H}^{+}}(H)\ar[r]^(0.6){\phi} & G}\in\cat{H}^{+}$,
there is a unique $\xymatrix{H\ar[r]^(0.4){\hat{\phi}} & N_{\cat{H}^{+}}^{\star}(G)}\in\cat{H}$ such that $\theta_{G}\circ N_{\cat{H}^{+}}\left(\hat{\phi}\right)=\phi$.
\end{thm}

\begin{proof}

Define $V\left(\hat{\phi}\right)(v):=V^{+}(\phi)(v)$ and
\[
E\left(\hat{\phi}\right)(e):=\left(
E^{+}(\phi)(e),
\left(\mathcal{P}V^{+}(\phi)\circ\epsilon_{G}\right)(e)
\right).
\]

\end{proof}

As a weak set-system homomorphism relaxes the commutative squares of traditional set-system homomorphisms,
there is no surprise that $\cat{H}^{+}$ can be represented as a lax comma category.
Indeed,
weak set-system homomorphisms are to traditional set-system homomorphisms
as metric contractions are to metric isometries.
However,
the ability to map an edge inside of another edge is reminiscent of incidence structures.
In truth,
all three of these categories are isomorphic.

\begin{thm}[Lax comma category characterization]\label{thm:laxcomma}
The categories $\istr$ and $\cat{H}^{+}$ are isomorphic to $\left(F^{\diamond}\backslash\backslash\Box^{\mathrm{T}}\mathcal{P}^{+}\right)$ as 2-categories.
\end{thm}

\begin{proof}

Let $\cat{C}:=\left(F^{\diamond}\backslash\backslash\Box^{\mathrm{T}}\mathcal{P}^{+}\right)$.
Since $\set$ is the domain 2-category of both $F^{\diamond}$ and $\Box^{\mathrm{T}}\mathcal{P}^{+}$,
the 2-cells of $\cat{C}$ are trivial.
Thus,
only the 0-cells/objects and 1-cells/morphisms are considered.

If $(Y,\epsilon,X)\in\ob(\cat{C})$,
then $X$ and $Y$ are sets,
and $\epsilon$ is a monotone map from $F^{\diamond}(Y)$ to $\left(\mathcal{P}^{+}(X)\right)^{\mathrm{T}}$,
meaning $\epsilon$ is a function from $Y$ to the power set of $X$.
Thus,
$(Y,\epsilon,X)\in\ob(\cat{H})$.
If $\xymatrix{(Y,\epsilon,X)\ar[rr]^{(g,\alpha,f)} & & \left(Y',\epsilon',X'\right)}\in\cat{C}$,
then $f$ is a function from $X$ to $X'$,
$g$ is a function from $Y$ to $Y'$,
and $\alpha$ is a 2-cell from $\epsilon'\circ F^{\diamond}(g)$ to $\left(\mathcal{P}^{+}(f)\right)^{\mathrm{T}}\circ\epsilon$ in $\qord\left(F^{\diamond}(Y),\left(\mathcal{P}^{+}(X)\right)^{\mathrm{T}}\right)$.
Therefore,
\[
\left(\epsilon'\circ F^{\diamond}(g)\right)(e)
\leq_{\left(\mathcal{P}^{+}(X)\right)^{\mathrm{T}}}
\left(\left(\mathcal{P}^{+}(f)\right)^{\mathrm{T}}\circ\epsilon\right)(e)
\]
for all $e\in Y$,
which translates to $\mathcal{P}(f)\left(\epsilon(e)\right)\subseteq\epsilon'\left(g(e)\right)$.
Hence,
\[
\xymatrix{
(Y,\epsilon,X)
\ar[rr]^{(g,f)}
&
&
\left(Y',\epsilon',X'\right)
}\in\cat{H}^{+}
\]
Define $\xymatrix{\cat{C}\ar[r]^{A} & \cat{H}^{+}\ar[r]^{B} & \cat{C}}$ by the following:
\begin{itemize}
\item $A(Y,\epsilon,X):=(Y,\epsilon,X)$,
$B(Y,\epsilon,X):=(Y,\epsilon,X)$,
\item $A(g,\alpha,f):=(g,f)$,
$B(g,f):=(g,\alpha,f)$,
\end{itemize}
where $\alpha$ is the unique 2-cell from $\epsilon'\circ F^{\diamond}(g)$ to $\left(\mathcal{P}^{+}(f)\right)^{\mathrm{T}}\circ\epsilon$.
A routine check shows that both $A$ and $B$ are functors,
and $B=A^{-1}$.

Define $\xymatrix{\cat{H}^{+}\ar[r]^{C} & \istr\ar[r]^{D} & \cat{H}^{+}}$ by the following:
\begin{itemize}
\item $C(G):=\left(V^{+}(G),E^{+},\left\{(v,e):v\in\epsilon_{G}(e)\right\}\right)$,
$D(H):=\left(\check{E}(H),\epsilon_{H},\check{V}(H)\right)$,
where $\epsilon_{H}(e):=\left\{v:(v,e)\in I(H)\right\}$;
\item $C(\phi):=\left(V^{+}(\phi),E^{+}(\phi)\right)$,
$D(\varphi):=\left(\check{E}(\varphi),\check{V}(\varphi)\right)$.
\end{itemize}
A routine check shows that both $C$ and $D$ are functors,
and $D=C^{-1}$.

\end{proof}

\begin{figure}
\begin{center}$\xymatrix{
&
&
\cat{R}
\ar@/_/[d]_{\mathsf{S}_{\cat{R}}}^{\dashv}\\
\cat{H}^{+}
\ar@/^/[d]^{\star}_{\dashv}
\ar@{<->}[r]_(0.35){\cong}
&
\left(F^{\diamond}\backslash\backslash\Box^{\mathrm{T}}\mathcal{P}^{+}\right)
\ar@{<->}[r]_(0.65){\cong}
&
\istr
\ar@/_/[u]_{\mathsf{N}_{\cat{R}}}\\
\cat{H}
\ar@/^/[u]
}$\end{center}
\caption{Set-system \& Incidence Hypergraphs}
\label{fig:diagram2}
\end{figure}

The functors of this section collectively build the diagram in Figure \ref{fig:diagram2}.
Observe that the incidence-forming functor $\xymatrix{\cat{H}\ar[r]^{\mathcal{I}} & \cat{R}}$ from \cite[Definition 3.3.1]{grilliette2023} can be factored as $\mathcal{I}=\mathsf{N}_{\cat{R}}CN_{\cat{H}^{+}}$.
Likewise,
the incidence-forgetting operator $\mathscr{F}$ from \cite[Definition 3.3.3]{grilliette2023} acts the same as $D\mathsf{S}_{\cat{R}}$ on objects.
Thus,
some immediate facts can be deduced about the inclusions.

\begin{cor}[Continuity]
The functor $N_{\cat{H}^{+}}$ is not continuous,
and $\mathsf{N}_{\cat{R}}$ is not cocontinuous.
No functor $\xymatrix{\cat{H}^{+}\ar[r]^{M} & \cat{H}}$ satisfies that $M(G)=G$ for all set system hypergraphs $G$.
\end{cor}

\begin{proof}

Since $C$ is an isomorphism and $\mathsf{N}_{\cat{R}}$ admits a left adjoint,
both are continuous.
If $N_{\cat{H}^{+}}$ is continuous,
then $\mathcal{I}$ would be continuous,
contradicting \cite[Lemma 3.3.2]{grilliette2023}.

Since $C$ is an isomorphism and $N_{\cat{H}^{+}}$ admits a right adjoint,
both are cocontinuous.
If $\mathsf{N}_{\cat{R}}$ is cocontinuous,
then $\mathcal{I}$ would be cocontinuous,
contradicting \cite[Lemma 3.3.2]{grilliette2023}.

Say a functor $\xymatrix{\cat{H}^{+}\ar[r]^{M} & \cat{H}}$ satisfies that $M(G)=G$ for all set-system hypergraphs $G$.
Then,
$MD\mathsf{S}_{\cat{R}}$ satisfies that $MD\mathsf{S}_{\cat{R}}(G)=\mathscr{F}(G)$ for all incidence hypergraphs $G$,
which contradicts \cite[Lemma 3.3.4]{grilliette2023}.

\end{proof}

The corollary above states why $\mathcal{I}$ fails to be continuous and cocontinuous.
It is a composite of a left adjoint and a right adjoint,
and mixing the two types yields a functor that cannot be either.
Moreover,
the nonexistence result illustrates that weak set-system homomorphisms are too distinct from traditional set-system homomorphisms for the former to be transformed into the latter without necessarily modifying the hypergraphs themselves.
Since weak set-system homomorphisms correspond to incidence hypergraph homomorphisms and encompass traditional set-system homomorphisms,
the corollary above indicates that the study of incidence hypergraphs subsumes,
but cannot be replaced by,
the study of set-system hypergraphs.

\section{Clique Replacement Revisited}\label{sec:cliquereplacement}

Having now studied a different view of graph homomorphisms,
the clique replacement graph can be revisited from a different perspective.
Rather than passing from $\cat{H}$ to $\gra$,
consider an analogous link between $\cat{R}$ and $\cat{Q}$.
Given a quiver,
one can naturally construct an incidence structure between the vertices and edges based upon the source and target functions.
Indeed,
the process manifests quite effortlessly from the structure of $\set$.
The construction is presented categorically,
heavily using the characterization of the disjoint union as the coproduct in $\set$.

\begin{defn}[Associated incidence hypergraph]
Given a quiver $Q$,
consider the diagram below,
where $\xymatrix{\vec{E}(Q)\ar[r]^(0.35){\varpi_{Q}^{0}} & \{0,1\}\times\vec{E}(Q) & \vec{E}(Q)\ar[l]_(0.35){\varpi_{Q}^{1}}}\in\set$ are the canonical inclusions.
\[\xymatrix{
&
&
\vec{E}(Q)
\ar[drr]^{id_{\vec{E}(Q)}}
\ar[d]_{\varpi_{Q}^{1}}
\ar[dll]_{\tau_{Q}}
\\
\vec{V}(Q)
&
&
\{0,1\}\times\vec{E}(Q)
\ar@{..>}[rr]^{\exists!\omega_{\check{U}(Q)}}
\ar@{..>}[ll]^{\exists!\varsigma_{\check{U}(Q)}}
&
&
\vec{E}(Q)
\\
&
&
\vec{E}(Q)
\ar[urr]_{id_{\vec{E}(Q)}}
\ar[u]_{\varpi_{Q}^{0}}
\ar[ull]^{\sigma_{Q}}
}\]
There is a unique $\xymatrix{\{0,1\}\times\vec{E}(Q)\ar[r]^(0.65){\varsigma_{\check{U}(Q)}}   &   \vec{V}(Q)}\in\set$ such that $\varsigma_{\check{U}(Q)}\circ\varpi_{Q}^{0}=\sigma_{Q}$ and $\varsigma_{\check{U}(Q)}\circ\varpi_{Q}^{1}=\tau_{Q}$,
and there is a unique $\xymatrix{\{0,1\}\times\vec{E}(Q)\ar[r]^(0.65){\omega_{\check{U}(Q)}}   &   \vec{E}(Q)}\in\set$ such that $\omega_{\check{U}(Q)}\circ\varpi_{Q}^{0}=\omega_{\check{U}(Q)}\circ\varpi_{Q}^{1}=id_{\vec{E}(Q)}$.  Concretely,
\begin{itemize}
\item $\varsigma_{\check{U}(Q)}(n,e)=\left\{\begin{array}{cc}
\sigma_{Q}(e),  &   n=0,\\
\tau_{Q}(e),  &   n=1,
\end{array}\right.$
\item $\omega_{\check{U}(Q)}(n,e)=e$.
\end{itemize}
Define the incidence hypergraph
\[
\check{U}(Q):=\left(\vec{V}(Q),\vec{E}(Q),\{0,1\}\times\vec{E}(Q),\varsigma_{\check{U}(Q)},\omega_{\check{U}(Q)}\right).
\]
Given $\xymatrix{Q\ar[r]^{\phi} & Q'}\in\cat{Q}$,
consider the diagram below.
\[\xymatrix{
\vec{E}(Q)
\ar[d]_{\vec{E}(\phi)}
\ar[r]^{\varpi_{Q}^{0}}
&
I\check{U}(Q)
\ar@{..>}[d]^{\exists!I\check{U}(\phi)}
&
\vec{E}(Q)
\ar[d]^{\vec{E}(\phi)}
\ar[l]_{\varpi_{Q}^{1}}
\\
\vec{E}\left(Q'\right)
\ar[r]_{\varpi_{Q'}^{0}}
&
I\check{U}\left(Q'\right)
&
\vec{E}\left(Q'\right)
\ar[l]^{\varpi_{Q}^{1}}
}\]
There is a unique $\xymatrix{I\check{U}\left(Q\right)\ar[r]^{I\check{U}\left(\phi\right)} & I\check{U}\left(Q'\right)}\in\set$ such that
$I\check{U}(\phi)\circ\varpi_{Q}^{n}=\varpi_{Q'}^{n}\circ\vec{E}(\phi)$ for $n=0,1$.
Concretely,
$I\check{U}(\phi)(n,e)=\left(n,\vec{E}(\phi)(e)\right)$.
Define $\check{U}(\phi):=\left(\vec{V}(\phi),\vec{E}(\phi),I\check{U}(\phi)\right)$.
A routine calculation shows that $\check{U}$ is a functor from $\cat{Q}$ to $\cat{R}$.
\end{defn}

Please note that $\check{U}$ treats the source and target of an edge differently,
tagging the former with 0 and the latter with 1.
Consequently,
a directed loop in a quiver becomes a 2-edge,
not a 1-edge,
in the resulting incidence hypergraph.

The functor $\check{U}$ admits a right adjoint,
which replaces an edge in an incidence hypergraph with a directed clique.
Again,
this operation arises naturally from the structure of $\set$,
using a kernel pair rather than a coproduct.

\begin{defn}[Clique replacement quiver]
Given an incidence hypergraph $G$,
let $\xymatrix{\vec{E}\vec{R}(G)\ar@/^/[r]^{p_{G}^{0}}\ar@/_/[r]_{p_{G}^{1}} & I(G)}\in\set$ be a kernel pair for $\omega_{G}$.
\[\xymatrix{
&
&
\vec{E}\vec{R}(G)
\ar[dl]_{p_{G}^{0}}
\ar[dr]^{p_{G}^{1}}
\ar@{}[dd]|-{\textrm{pullback}}
\\
\check{V}(G)
&
I(G)
\ar[l]_{\varsigma_{G}}
\ar[dr]_{\omega_{G}}
&
&
I(G)
\ar[r]^{\varsigma_{G}}
\ar[dl]^{\omega_{G}}
&
\check{V}(G)
\\
&
&
\check{E}(G)
}\]
Concretely,
\begin{itemize}
\item $\vec{E}\vec{R}(G)=\left\{(i,j):\omega_{G}(i)=\omega_{G}(j)\right\}$,
\item $p_{G}^{0}(i,j)=i$, $p_{G}^{1}(i,j)=j$.
\end{itemize}
Define $\vec{R}(G):=\left(\check{V}(G),\vec{E}\vec{R}(G),\varsigma_{G}\circ p_{G}^{0},\varsigma_{G}\circ p_{G}^{1}\right)$.
Consider the diagram below.
\[\xymatrix{
\vec{E}\vec{R}(G)
\ar[rr]^{\varpi_{\vec{R}(G)}^{0}}
\ar[drr]_{p_{G}^{0}}
&
&
I\check{U}\vec{R}(G)
\ar@{..>}[d]^{\exists!I\left(\check{\theta}_{G}\right)}
&
&
\vec{E}\vec{R}(G)
\ar[ll]_{\varpi_{\vec{R}(G)}^{1}}
\ar[dll]^{p_{G}^{1}}
\\
&
&
I(G)
}\]
There is a unique $\xymatrix{I\check{U}\vec{R}(G)\ar[rr]^{I\left(\check{\theta}_{G}\right)} & & I(G)}\in\set$ such that $I\left(\check{\theta}_{G}\right)\circ\varpi_{\vec{R}(G)}^{n}=p_{G}^{n}$ for $n=0,1$.
Concretely,
\[
I\left(\check{\theta}_{G}\right)\left(n,(i,j)\right)=\left\{\begin{array}{rcl}
i,  &   n=0,\\
j,  &   n=1.
\end{array}\right.
\]
Define $\check{\theta}_{G}:=\left(id_{\check{V}(G)},\omega_{G}\circ p_{G}^{0},I\left(\check{\theta}_{G}\right)\right)$.
One can check that $\xymatrix{\check{U}\vec{R}(G)\ar[r]^(0.6){\check{\theta}_{G}} & G}\in\cat{R}$.
\end{defn}

\begin{thm}[Right adjoint characterization]
Given $\xymatrix{\check{U}(Q)\ar[r]^(0.6){\phi} & G}\in\cat{R}$,
there is a unique $\xymatrix{Q\ar[r]^(0.4){\hat{\phi}} & \vec{R}(G)}\in\cat{Q}$ such that
$\check{\theta}_{G}\circ\check{U}\left(\hat{\phi}\right)=\phi$.
\end{thm}

\begin{proof}

Consider the diagram below.
\[\xymatrix{
I\check{U}(Q)
\ar[d]_{I(\phi)}
&
\vec{E}(Q)
\ar[l]_{\varpi_{Q}^{0}}
\ar[r]^{\varpi_{Q}^{1}}
&
I\check{U}(Q)
\ar[d]^{I(\phi)}
\\
I(G)
\ar[dr]_{\omega_{G}}
&
&
I(G)\ar[dl]^{\omega_{G}}
\\
&
\check{E}(G)
}\]
As $\phi$ is an incidence hypergraph homomorphism,
one has
\[\begin{array}{rcl}
\omega_{G}\circ I(\phi)\circ\varpi_{Q}^{0}
&   =   &   \check{E}(\phi)\circ\omega_{\check{U}(Q)}\circ\varpi_{Q}^{0}
=\check{E}(\phi)\circ id_{\vec{E}(Q)}
=\check{E}(\phi)\circ\omega_{\check{U}(Q)}\circ\varpi_{Q}^{1}\\
&   =   &   \omega_{G}\circ I(\phi)\circ\varpi_{Q}^{1}.
\end{array}\]
There is a unique $\xymatrix{\vec{E}(Q)\ar[r]^{\vec{E}\left(\hat{\phi}\right)} & \vec{E}\vec{R}(G)}\in\set$ such that $p_{G}^{n}\circ\vec{E}\left(\hat{\phi}\right)=I(\phi)\circ\varpi_{Q}^{n}$ for $n=0,1$.
Define $\hat{\phi}:=\left(\check{V}(\phi),\vec{E}\left(\hat{\phi}\right)\right)$.

\end{proof}

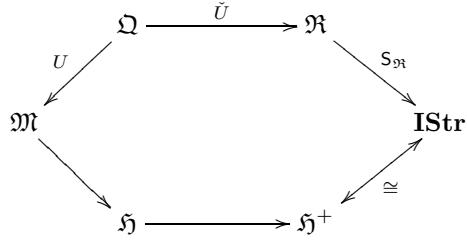
\begin{figure}
\begin{center}$\xymatrix{
&   \cat{Q}\ar[rr]^{\check{U}}\ar[dl]_{U} &   &   \cat{R}\ar[dr]^{\mathsf{S}_{\cat{R}}}\\
\cat{M}\ar[dr] &   &   &   &   \istr\\
&   \cat{H}\ar[rr] &   &   \cat{H}^{+}\ar@{<->}[ur]_{\cong}
}$\end{center}
\caption{Preservation of Vertices and Edges}
\label{fig:hexagon}
\end{figure}

With $\check{U}$ and $\vec{R}$ in hand,
consider Figure \ref{fig:hexagon}.
Observe that all of the functors in Figure \ref{fig:hexagon} leave vertices and edges unchanged,
merely changing how incidence and adjacency are represented.
Consequently,
if $\xymatrix{\cat{H}^{+}\ar[r]^{C} & \istr}$ is the isomorphism from Theorem \ref{thm:laxcomma},
the hexagon commutes,
meaning corresponding hexagon of right adjoints commutes up to isomorphism.

\begin{thm}[Coherence of clique replacement]\label{thm:cliquereplacement}
The hexagon in Figure \ref{fig:hexagon} commutes:
$\mathsf{S}_{\cat{R}}\check{U}
=
CN_{\cat{H}^{+}}N_{\cat{H}}U$.
Consequently,
the following isomorphism is natural for all set-system hypergraphs $G$.
\begin{center}$
\vec{R}\mathsf{N}_{\cat{R}}C(G)
\cong
\vec{D}\del N_{\cat{H}^{+}}^{\star}(G)
$\end{center}
\end{thm}

Effectively,
this theorem states that given a set-system hypergraph,
the following two constructions yield isomorphic quivers:
\begin{enumerate}

\item $\vec{D}\del N_{\cat{H}^{+}}^{\star}$
\begin{enumerate}
\item replace each edge with an abstract simplicial complex,
\item remove all nontraditional edges,
\item replace each 1-edge with a directed cycle and each 2-edge with a directed 2-cycle;
\end{enumerate}

\item $\vec{R}\mathsf{N}_{\cat{R}}C$
\begin{enumerate}
\item convert into an incidence structure,
\item view as an incidence hypergraph,
\item replace each edge with a directed clique.
\end{enumerate}

\end{enumerate}

Since both $\vec{D}$ and $\vec{R}$ produce the equivalent of an undirected graph,
the above theorem suggests what an undirected clique replacement functor should be.

\begin{defn}[Factored clique replacement]
Define the composite functor $R:=\mathsf{S}_{\cat{M}}\del N_{\cat{H}^{+}}^{\star}N_{\cat{H}^{+}}$.
\end{defn}

Tracing through each construction yields a concrete represention of the action of $R$,
which is precisely the action of the classical clique replacement graph $\Gamma$ without the artificial removal of 1-edges.

\begin{lem}[Action of $R$]
If $\xymatrix{G\ar[r]^{\phi} & G'}\in\cat{H}$,
then
\begin{itemize}
\item $VR(G)=V(G)$,
\item $ER(G)=\left\{\{v,w\}:\exists e\in E(G)\left(\{v,w\}\subseteq\epsilon_{G}(e)\right)\right\}$,
\item $R(\phi)=V(\phi)$.
\end{itemize}
\end{lem}

If $\xymatrix{\gra\ar[r]^(0.4){Z_{\gra}} & \sdigra}$ is the isomorphism from Figure \ref{fig:diagram},
then Theorem \ref{thm:compatibility} and the result above provide an alternate factorization of $R$,
which explicitly shows it to be $\vec{R}$ under layers of converting between categories.
Hence,
the two operations legitimately correspond to each other,
merely in two different contexts.

\begin{prop}[Alternate factorization through $\cat{R}$]
The functor $R$ can be factored as $
Z_{\gra}^{-1}N_{\digra}^{\diamond}\mathsf{S}_{\cat{Q}}\vec{R}\mathsf{N}_{\cat{R}}CN_{\cat{H}^{+}}
$.
\end{prop}

\begin{proof}

Observe that $N_{\digra}^{\diamond}N_{\digra}=id_{\sdigra}$,
so
\[\begin{array}{rcl}
R
&   =   &   Z_{\gra}^{-1}Z_{\gra}R\\
&   =   &   Z_{\gra}^{-1}N_{\digra}^{\diamond}N_{\digra}Z_{\gra}R\\
&   =   &   Z_{\gra}^{-1}N_{\digra}^{\diamond}N_{\digra}Z_{\gra}\mathsf{S}_{\cat{M}}\del N_{\cat{H}^{+}}^{\star}N_{\cat{H}^{+}}\\
&   =   &   Z_{\gra}^{-1}N_{\digra}^{\diamond}\mathsf{S}_{\cat{Q}}\vec{D}\del N_{\cat{H}^{+}}^{\star}N_{\cat{H}^{+}}\\
&   =   &   Z_{\gra}^{-1}N_{\digra}^{\diamond}\mathsf{S}_{\cat{Q}}\vec{R}\mathsf{N}_{\cat{R}}CN_{\cat{H}^{+}}
\end{array}\]
by Theorems \ref{thm:compatibility} and \ref{thm:cliquereplacement}.

\end{proof}

Considering how obfuscated $R$ is from $\vec{R}$,
one might expect that like $\mathcal{I}$ from \cite[Definition 3.3.1]{grilliette2023},
$R$ would not admit an adjoint.
However,
$R$ does indeed admit a right adjoint,
given by a closure operation that creates abstract simplicial complexes.
Considering that $N_{\cat{H}^{+}}$ is involved in both factorizations of $R$,
and the right adjoint of $N_{\cat{H}^{+}}$ is the simplicial replacement,
this characterization is likely not coincidental.

\begin{defn}[Simplicial closure]
Given a graph $G$,
define the set-system hypergraph $R^{\star}(G)$ by
\begin{itemize}
\item $VR^{\star}(G):=V(G)$,
\item $ER^{\star}(G):=\left\{A\in\mathcal{P}V(G):\forall v,w\in A\left(\{v,w\}\in E(G)\right)\right\}$,
\item $\epsilon_{R^{\star}(G)}(A):=A$.
\end{itemize}
Define $\xymatrix{RR^{\star}(G)\ar[r]^(0.65){\theta_{G}} & G}\in\gra$ by $\theta_{G}(v):=v$.
\end{defn}

\begin{thm}[Universal property]
Given $\xymatrix{R(H)\ar[r]^{f} & G}\in\gra$,
there is a unique $\xymatrix{H\ar[r]^{\hat{f}} & R^{\star}(G)}\in\cat{H}$ such that $\theta_{G}\circ R\left(\hat{f}\right)=f$.
\end{thm}

\begin{proof}

Define $\hat{f}:=\left(E\left(\hat{f}\right),f\right)$,
where $E\left(\hat{f}\right)(e):=\left(\mathcal{P}(f)\circ\epsilon_{H}\right)(e)$.

\end{proof}

\section{Duality \& Intersection Revisited}\label{sec:duality}

Much like how clique replacement was recovered by passing through weak set-system homomorphisms,
the intersection graph can be as well.
Indeed,
the intersection graph is meant to serve as a duality,
converting edges to vertices.
Sadly,
it is not truly a dual as it is not invertible.

On the other hand,
$\cat{R}$ has incidence duality,
which is self-inverting and is deeply related to the Laplacian product introduced in \cite{ih3}.
Recall from \cite[Lemma 3.1.2]{ih3} that the incidence dual $\xymatrix{\cat{R}\ar[r]^{\Box^{\#}} & \cat{R}}$ acts in the following way:
\begin{itemize}
\item $G^{\#}=\left(\check{E}(G),\check{V}(G),I(G),\omega_{G},\varsigma_{G}\right)$,
\item $\phi^{\#}=\left(\check{E}(\phi),\check{V}(\phi),I(\phi)\right)$.
\end{itemize}

Through the simplification adjunction on $\cat{R}$,
incidence duality can be conjugated to act on $\istr$,
as well as $\cat{H}^{+}$ through the isomorphism in Theorem \ref{thm:laxcomma}.

\begin{defn}[Incidence dual for $\istr$ \& $\cat{H}^{+}$]
Define $\Box^{\top}:=S_{\cat{R}}\Box^{\#}\mathsf{I}_{\cat{R}}$
and $\Box^{\ddag}:=C^{-1}\Box^{\top}C$,
where $\xymatrix{\cat{H}^{+}\ar[r]^{C} & \istr}$ is the isomorphism from Theorem \ref{thm:laxcomma}.
\end{defn}

Direct calculation shows the concrete action of both dualities and that both are self-inverting.
Notably,
$\Box^{\top}$ is precisely the dual structure of an incidence structure \cite{beth,dembowski},
demonstrating that $\Box^{\#}$ generalizes $\Box^{\top}$ to allow parallel incidences.
Moreover,
$\Box^{\ddag}$ is the dual hypergraph of a set-system hypergraph \cite[p.\ 187]{dorfler1980},
but without the artificial restriction to avoid isolated vertices.

\begin{lem}[Action of $\Box^{\top}$ \& $\Box^{\ddag}$]
If $\xymatrix{G\ar[r]^{\phi} & G'}\in\istr$,
then
\begin{itemize}
\item $G^{\top}=\left(\check{E}(G),\check{V}(G),\left\{(e,v):(v,e)\in I(G)\right\}\right)$,
\item $\phi^{\top}=\left(\check{E}(\phi),\check{V}(\phi)\right)$.
\end{itemize}
If $\xymatrix{H\ar[r]^{\varphi} & H'}\in\cat{H}^{+}$,
then
\begin{itemize}
\item $G^{\ddag}=\left(E(G),V(G),\epsilon_{G^{\ddag}}\right)$,
where $\epsilon_{G^{\ddag}}(v):=\left\{e:v\in\epsilon_{G}(e)\right\}$;
\item $\varphi^{\ddag}=\left(E(\varphi),V(\varphi)\right)$.
\end{itemize}
\end{lem}

However,
notice that $\Box^{\ddag}$ is an autofunctor on $\cat{H}^{+}$,
not $\cat{H}$ like \cite[p.\ 187]{dorfler1980} claims.
Proposition \ref{prop:duality} demonstrates that $\Box^{\ddag}$ cannot be restricted to $\cat{H}$ and remain functorial.
Indeed,
incidence duality of this type appears to require more than what traditional graph homomorphisms can allow.

Now,
observe what happens if $\Box^{\ddag}$ is inserted into the factorization of $R$.
Direct calculation shows that the resulting functor recovers the action of the classical intersection graph $L$ without the artificial removal of 1-edges.

\begin{defn}[Factored intersection]
Define the composite functor
\[
\Lambda:=\mathsf{S}_{\cat{M}}\del N_{\cat{H}^{+}}^{\star}\Box^{\ddag}N_{\cat{H}^{+}}.
\]
\end{defn}

\begin{lem}[Action of $\Lambda$]
If $\xymatrix{G\ar[r]^{\phi} & G'}\in\cat{H}$,
then
\begin{itemize}
\item $V\Lambda(G):=E(G)$,
\item $E\Lambda(G):=\left\{\{e,f\}:\epsilon_{G}(e)\cap\epsilon_{G}(f)\neq\emptyset\right\}$.
\item $\Lambda(\phi)=E(\phi)$.
\end{itemize}
\end{lem}

\printbibliography

\end{document}